\documentclass[12pt, reqno]{amsart}
\usepackage{amsmath, amsthm, amscd, amsfonts, amssymb, graphicx, color}
\usepackage[bookmarksnumbered, colorlinks, plainpages]{hyperref}

\input{mathrsfs.sty}

\newtheorem{theorem}{Theorem}[section]
\newtheorem{lemma}[theorem]{Lemma}
\newtheorem{proposition}[theorem]{Proposition}
\newtheorem{corollary}[theorem]{Corollary}
\theoremstyle{definition}

\theoremstyle{remark}
\newtheorem{remark}[theorem]{Remark}
\numberwithin{equation}{section}
\begin{document}

\title [Complete refinments of the Berezin number inequalities]{Complete refinements of the Berezin number inequalities}

\author[M. Bakherad, R. Lashkaripour, M. Hajmohamadi and U. Yamanci ]{M. Bakherad$^1$, R. Lashkaripour$^2$, M. Hajmohamadi$^3$, and U. Yamanci$^4$ }

\address{$^1$$^{,2}$$^{,3}$ Department of Mathematics, Faculty of Mathematics, University of Sistan and Baluchestan, Zahedan, I.R.Iran.\\ $^4$ Department of Statistics, Suleyman Demirel University, 32260, Isparta, Turkey\\
Current Adress: Department of Mathematics and Statistics, University of
Toledo, Toledo, OH 43606, USA.}

\email{$^{1}$mojtaba.bakherad@yahoo.com; bakherad@member.ams.org}

\email{$^2$lashkari@hamoon.usb.ac.ir}

\email{$^{3}$monire.hajmohamadi@yahoo.com}

\email{$^{4}$ulasyamanci@sdu.edu.tr }


\begin{abstract}
In this paper, several refinements of the Berezin number inequalities are obtained. We generalize inequalities involving powers of the Berezin number for product of two operators acting on a reproducing kernel Hilbert space $\mathcal H=\mathcal H(\Omega)$ and also improve them. Among other inequalities, it is shown that if $A,B\in {\mathcal B}(\mathcal H)$ such that $|A|B=B^{*}|A|$, $f$ and $g$ are nonnegative continuous functions on $[0,\infty)$ satisfying $f(t)g(t)=t\,(t\geq 0)$, then
\begin{align*}
&\textbf{ber}^{p}(AB)\leq r^{p}(B)\times\\&\left(\textbf{ber} \big(\frac{1}{\alpha}f^{\alpha p}(|A|)+\frac{1}{\beta}g^{\beta p}(|A^{*}|)\big)-r_{0}\big(\langle f^{2}(|A|)\hat{k}_{\lambda},\hat{k}_{\lambda}\rangle^{\alpha p/4} -\langle g^{2}(|A^{*}|)\hat{k}_{\lambda},\hat{k}_{\lambda}\rangle^{\beta p/4}\big)^{2}\right)
\end{align*}
for every $p\geq 1, \alpha\geq\beta>1$ with $\frac{1}{\alpha}+\frac{1}{\beta}=1$, $\beta p\geq2$ and $r_{0}=\min\{\frac{1}{\alpha},\frac{1}{\beta}\}$.

\end{abstract} \maketitle
\section{Introduction}
  Throughout this paper, a reproducing kernel Hilbert space (RKHS for short)  $\mathcal H=\mathcal H(\Omega)$ is a Hilbert space of complex valued functions on a (nonempty) set $\Omega$, which has the property that point evaluations are continuous i.e. for each $\lambda\in \Omega$ the map $f\mapsto f(\lambda)$ is a continuous linear functional on $\mathcal H$. The Riesz representation theorem ensure that for each $\lambda\in \Omega$ there is a unique element $k_{\lambda}\in \mathcal H$ such that $f(\lambda)=\langle f,k_{\lambda}\rangle$, for all $f\in \mathcal H$. The collection $\{k_{\lambda} : \lambda\in \Omega\}$ is called the reproducing kernel of $\mathcal H$. If $\{e_{n}\}$ is an orthonormal basis for a functional Hilbert space $\mathcal H$, then the reproducing kernel of $\mathcal H$ is given by $k_{\lambda}(z)=\sum_n\overline{e_{n}(\lambda)}e_{n}(z)$; (see \cite[problem 37]{ando}). For $\lambda\in \Omega$, let $\hat{k_{\lambda}}=\frac{k_{\lambda}}{\|k_{\lambda}\|}$ be the normalized reproducing kernel of $\mathcal H$. For a bounded linear operator $A$ on $\mathcal H$, the function $\widetilde{A}$ defined on $\Omega$ by $\widetilde{A}(\lambda)=\langle A\hat{k_{\lambda}},\hat{k_{\lambda}}\rangle$ is the Berezin symbol of $A$, which firstly have been introduced by Berezin \cite{Ber1,Ber2}. The Berezin set and the Berezin number of the operator A are defined by
 \begin{align*}
 \textbf{Ber}(A):=\{\widetilde{A}(\lambda): \lambda\in \Omega\} \qquad \textrm{and} \qquad \textbf{ber}(A):=\sup\{|\widetilde{A}(\lambda)|: \lambda\in\Omega\},
 \end{align*}
respectively(see \cite{kar}).
 The Berezin number of  operators $A$ and $B$ satisfies the property
$\textbf{ ber}(\alpha A)=|\alpha|\textbf{ber}(A)\,\,(\alpha\in \mathbb C)$ and
$\textbf{ber}(A+B)\leq \textbf{ber}(A)+\textbf{ber}(B)$ and $\textbf{ber}(A)\leq\|A\|$, where $\|\,\cdot\,\|$ is the operator norm.
The spectral radius of  $A\in\mathcal B({\mathcal H})$ is defined by $r(A):=\sup\{|\eta| : \eta\in sp(A)\}$. Let
$$l(A)=\inf\big\{\|Ax\|:x\in {\mathcal H}, \|x\|=1\big\}=\inf\big\{|\langle Ax,y\rangle| : x,y\in {\mathcal H}, \|x\|=\|y\|=1\big\}.$$
 In \cite{kit1}, Kittaneh estimated a spectral radius inequality for any $A,B\in \mathcal{B}(\mathcal H)$ as follows:
\begin{align}\label{ki1}
r(AB)\leq \frac{1}{4}\left(\|AB\|+\|BA\|+\sqrt{(\|AB\|-\|BA\|)^{2}+4m(A,B)}\right),
\end{align}
where $m(A,B)=\min\{\|A\|\|BAB\|,\|B\|\|ABA\|\}$.
Also, he showed
\begin{align}\label{ki3}
\|A^{1/2}B^{1/2}\|\leq\|AB\|^{1/2}
\end{align}
and
\begin{align}\label{ber1}
\|A+B\|\leq\frac{1}{2}(\|A\|+\|B\|+\sqrt{(\|A\|-\|B\|)^{2}+4\min(\|AB\|,\|BA\|)}).
\end{align}
Namely, the Berezin transform have been investigated in detail for the Toeplitz and Hankel operators on the Hardy and Bergman spaces; it is widely applied in the various questions of analysis and uniquely determines the operator(i.e., for all $\lambda\in \Omega,  \widetilde{A}(\lambda)=\widetilde{B}(\lambda)$ implies $A=B$). For further information about Berezin symbol we refer the
reader to \cite{Ba,kar1,kar2,Nor} and references therein. Recently in \cite{Ba1,haj, haj1,haj2, ygg, yg} have studied about the inequalities for the Berezin number and the numerical radius of operators. Also, some Berezin number inequalities were obtained by using the Hardy types inequalities(see \cite{kar5, ggs, ygc}).
\section{main results}
To prove our Berezin number inequalities, we need several well known lemmas.
\begin{lemma}\cite{kit}
Let $A,B\in {\mathcal B}(\mathcal H)$ such that $|A|B=B^{*}|A|$. If $f$ and $g$ are nonnegative continuous functions on $[0,\infty)$ satisfying $f(t)g(t)=t\,(t\geq 0)$, then
\begin{align}\label{al1}
|\langle AB x,y\rangle|\leq r(B)\|f(|A|)x\|\|g(|A^{*}|)y\|
\end{align}
for every $x,y\in \mathcal H$.
\end{lemma}
\begin{lemma}\cite{alo}
Let $A\in {\mathcal B}(\mathcal H)$ be positive. Then
\begin{align}\label{al2}
|\langle Ax,x\rangle|^{2p}&\leq [\langle A^{p}x,x\rangle-\langle |A-\langle Ax,x\rangle I|^{p}x,x\rangle]\times [\langle A^{p}y,y\rangle-\langle |A-\langle Ay,y\rangle I|^{p}y,y\rangle]\nonumber\\&
\leq\langle A^{p}x,x\rangle\langle A^{p}y,y\rangle
\end{align}
for all $p\geq2$ and any $x,y\in \mathcal H$.
\end{lemma}
Now, we show some Berezin number inequalities.
\begin{theorem}\label{th1}
Let $A,B\in {\mathcal B}(\mathcal H)$ such that $|A|B=B^{*}|A|$. If $f$ and $g$ are nonnegative continuous functions on $[0,\infty)$ satisfying $f(t)g(t)=t\,(t\geq 0)$, then
\begin{align*}
\textbf{ber}(AB)&\leq\frac{1}{2}r(B)\textbf{ber}[f^{2}(|A|)+g^{2}(|A^{*}|)].
\end{align*}
In particular for $f(t)=t^\alpha$ and $g(t)=t^{1-\alpha}$, we have
\begin{align*}
\textbf{ber}(AB)&\leq\frac{1}{2}r(B)\textbf{ber}(|A|^{2\alpha}+|A^{*}|^{2(1-\alpha)}).
\end{align*}
\end{theorem}
\begin{proof}
If we put $x=\hat{k}_{\lambda}$ in \eqref{al1}, we have
\begin{align}
|\langle AB \hat{k}_{\lambda},\hat{k}_{\lambda}\rangle|&\leq r(B)\|f(|A|)\hat{k}_{\lambda}\|\|g(|A^{*}|)\hat{k}_{\lambda}\|\nonumber\\&
=r(B)\langle f^{2}(|A|)\hat{k}_{\lambda},\hat{k}_{\lambda}\rangle^{1/2}\langle g^{2}(|A^{*}|)\hat{k}_{\lambda},\hat{k}_{\lambda}\rangle^{1/2}\nonumber\\&
\leq\frac{1}{2}r(B)(\langle f^{2}(|A|)\hat{k}_{\lambda},\hat{k}_{\lambda}\rangle+\langle g^{2}(|A^{*}|)\hat{k}_{\lambda},\hat{k}_{\lambda}\rangle\nonumber\\&
=\frac{1}{2}r(B)\langle (f^{2}(|A|)+g^{2}(|A^{*}|))\hat{k}_{\lambda},\hat{k}_{\lambda}\rangle\nonumber\\&
\leq\frac{1}{2}r(B)\textbf{ber}[f^{2}(|A|)+g^{2}(|A^{*}|)].
\end{align}
By taking the supremum over $\lambda\in\Omega$ we have
\begin{align*}
\textbf{ber}(AB)\leq\frac{1}{2}r(B)\textbf{ber}[f^{2}(|A|)+g^{2}(|A^{*}|)].
\end{align*}
\end{proof}

\begin{remark}\label{th1}
Let $A,B\in {\mathcal B}(\mathcal H)$ such that $|A|B=B^{*}|A|$. If $f$ and $g$ are nonnegative continuous functions on $[0,\infty)$ satisfying $f(t)g(t)=t\,(t\geq 0)$, then from \eqref{ber1} and \eqref{ki1} for $A=I$, we have
\begin{align*}
\textbf{ber}(AB)&\leq\frac{1}{2}r(B)\textbf{ber}[f^{2}(|A|)+g^{2}(|A^{*}|)]\\&
\leq \frac{1}{8}(\|B\|+\|B^{2}\|^{1/2})\big[\|f^{2}(|A|)\|+\|g^{2}(|A^{*}|)\|\\&
\qquad\qquad+\sqrt{(\|f^{2}(|A|)\|-\|g^{2}(|A^{*}|)\|)^{2}+4\|f(|A|).g(|A^{*}|\|^{2}}\big].
\end{align*}
\end{remark}
\begin{theorem}\label{th5}
Let $A,B\in {\mathcal B}({\mathcal H})$ such that $|A|B=B^{*}|A|$. If $f$ and $g$ are nonnegative continuous functions on $[0,\infty)$ satisfying $f(t)g(t)=t\,(t\geq 0)$, then
\begin{align*}
\textbf{ber}^{p}(AB)&\leq r^{p}(B)\textbf{ber}\left[ \frac{1}{\alpha}f^{\alpha p}(|A|)+\frac{1}{\beta}g^{\beta p}(|A^{*}|)\right]
\end{align*}
for every $p\geq 1, \alpha\geq\beta>1$ with $\frac{1}{\alpha}+\frac{1}{\beta}=1$ and $\beta p\geq2$.
\end{theorem}
\begin{proof}
Let $\hat{k}_{\lambda}\in {\mathcal H}$. We have
\begin{align*}
|\langle AB\hat{k}_{\lambda},\hat{k}_{\lambda}\rangle|^{p}&\leq r^{p}(B)\|f(|A|)\hat{k}_{\lambda}\|^{p}\|g(|A^{*}|)\hat{k}_{\lambda}\|^{p}\\&
=r^{p}(B)\langle f^{2}(|A|)\hat{k}_{\lambda},\hat{k}_{\lambda}\rangle^{p/2}\langle g^{2}(|A^{*}|)\hat{k}_{\lambda},\hat{k}_{\lambda}\rangle^{p/2}\\&
\leq r^{p}(B)\left[\frac{1}{\alpha}\langle f^{2}(|A|)\hat{k}_{\lambda},\hat{k}_{\lambda}\rangle^{\alpha p/2}+\frac{1}{\beta}\langle g^{2}(|A^{*}|)\hat{k}_{\lambda},\hat{k}_{\lambda}\rangle^{\beta p/2}\right]\\&
\leq r^{p}(B)\left[\frac{1}{\alpha}\langle f^{\alpha p}(|A|)\hat{k}_{\lambda},\hat{k}_{\lambda}\rangle+\frac{1}{\beta}\langle g^{\beta p}(|A^{*}|)\hat{k}_{\lambda},\hat{k}_{\lambda}\rangle\right]\\&
=r^{p}(B)\langle [\frac{1}{\alpha}f^{\alpha p}(|A|)+\frac{1}{\beta}g^{\beta p}(|A^{*}|)]\hat{k}_{\lambda},\hat{k}_{\lambda}\rangle.\\&
\leq r^{p}(B)\textbf{ber}(\frac{1}{\alpha}f^{\alpha p}(|A|)+\frac{1}{\beta}g^{\beta p}(|A^{*}|)).
\end{align*}
By taking the supremum over $\lambda\in\Omega$ we get the desired result.
\end{proof}
In the following by using of refinements of the Cauchy-Schwarz inequality, we have an upper bound for product two operators.
\begin{proposition}
Let $A,B\in {\mathcal B}({\mathcal H})$ such that $|A|B=B^{*}|A|$. If $f$ and $g$ are nonnegative continuous functions on $[0,\infty)$ satisfying $f(t)g(t)=t\,(t\geq 0)$, then
\begin{align}\label{al3}
|\langle AB \hat{k}_{\lambda},\hat{k}_{\mu}\rangle|&\leq r(B)
\sqrt[2p]{\langle f^{2p}(|A|)\hat{k}_{\lambda},\hat{k}_{\lambda}\rangle-\langle |f^{2}(|A|)-\langle f^{2}(|A|)\hat{k}_{\lambda},\hat{k}_{\lambda}\rangle I|^{p}\hat{k}_{\lambda},\hat{k}_{\lambda}\rangle}\nonumber\\&\times
\sqrt[2p]{\langle g^{2p}(|A^{*}|)\hat{k}_{\mu},\hat{k}_{\mu}\rangle-\langle |g^{2}(|A^{*}|)-\langle g^{2}(|A^{*}|)\hat{k}_{\mu},\hat{k}_{\mu}\rangle I|^{p}\hat{k}_{\mu},\hat{k}_{\mu}\rangle}\nonumber\\&
\leq r(B) \sqrt[2p]{\langle f^{2p}(|A|)\hat{k}_{\lambda},\hat{k}_{\lambda}\rangle}\sqrt[2p]{\langle g^{2p}(|A^{*}|)\hat{k}_{\mu},\hat{k}_{\mu}\rangle}
\end{align}
for all $p\geq2$ and any $\hat{k}_{\lambda},\hat{k}_{\mu}\in {\mathcal H}$.
\end{proposition}
\begin{proof}
Let $\hat{k}_{\lambda},\hat{k}_{\mu}\in {\mathcal H}$. Applying \eqref{al1} and \eqref{al2}, we have
\begin{align*}
|\langle AB \hat{k}_{\lambda},\hat{k}_{\mu}\rangle|&\leq r(B)\|f(|A|)\hat{k}_{\lambda}\|\|g(|A^{*}|)\hat{k}_{\mu}\|\\&
\leq r(B)\langle f(|A|)\hat{k}_{\lambda}^{2},\hat{k}_{\lambda}\rangle^{1/2}\langle g(|A^{*}|)\hat{k}_{\mu}^{2},\hat{k}_{\mu}\rangle^{1/2}\\&
\leq r(B) \sqrt[2p]{\langle f(|A|)\hat{k}_{\lambda}^{2p},\hat{k}_{\lambda}\rangle-\langle |f^{2}(|A|)-\langle f^{2}(|A|)\hat{k}_{\mu},\hat{k}_{\mu}\rangle I|\hat{k}_{\mu},\hat{k}_{\mu}\rangle}\\&\times
\sqrt[2p]{\langle g(|A^{*}|)\hat{k}_{\mu}^{2p},\hat{k}_{\mu}\rangle-\langle |g^{2}(|A^{*}|)-\langle g^{2}(|A^{*}|)\hat{k}_{\mu},\hat{k}_{\mu}\rangle I|\hat{k}_{\mu},\hat{k}_{\mu}\rangle}.
\end{align*}
We get the result.
\end{proof}
\begin{corollary}
Let $A,B\in {\mathcal B}({\mathcal H})$ such that $|A|B=B^{*}|A|$, $p\geq2$ and $0\leq\alpha  \leq1$. Then
\begin{align}
|\langle AB \hat{k}_{\lambda},\hat{k}_{\mu}\rangle|&\leq r(B)
\sqrt[2p]{\langle |A|^{2p\alpha}\hat{k}_{\lambda},\hat{k}_{\lambda}\rangle-\langle |A|^{2\alpha}-\langle | |A|^{2\alpha}\hat{k}_{\lambda},\hat{k}_{\lambda}\rangle I|^{p}\hat{k}_{\lambda},\hat{k}_{\lambda}\rangle}\nonumber\\&\times
\sqrt[2p]{\langle |A^{*}|^{2p(1-\alpha)}\hat{k}_{\mu},\hat{k}_{\mu}\rangle-\langle ||A^{*}|^{2(1-\alpha)}-\langle |A^{*}|^{2(1-\alpha)}\hat{k}_{\mu},\hat{k}_{\mu}\rangle I|^{p}\hat{k}_{\mu},\hat{k}_{\mu}\rangle}\nonumber\\&
\leq r(B) \sqrt[2p]{\langle |A|^{2p\alpha}\hat{k}_{\lambda},\hat{k}_{\lambda}\rangle}\sqrt[2p]{\langle |A^{*}|^{2p(1-\alpha)}\hat{k}_{\mu},\hat{k}_{\mu}\rangle}.
\end{align}
\end{corollary}
\begin{proof}
By putting $f(t)=t^{\alpha}$ and $g(t)=t^{1-\alpha}\,\,(0\leq\alpha  \leq1)$ in \eqref{al3}, we get the result.
\end{proof}
The next result gives an upper  bound for the product of two operators based on the
refinement of the Cauchy-Schwarz inequality.
\begin{theorem}
Let $A,B\in {\mathcal B}(\mathcal H)$ such that $|A|B=B^{*}|A|$. If $f$ and $g$ are nonnegative continuous functions on $[0,\infty)$ satisfying $f(t)g(t)=t\,(t\geq 0)$. Then
\begin{align*}
\textbf{ber}(AB)&\leq \frac{1}{2}(\|B\|+\|B^{2}\|^{1/2})\left[ \textbf{ber}(f^{2p}(|A|))-l\left(|[f^{2}(|A|)-\|f(|A|)\|^{2}]|^{p}\right)\right]^{\frac{1}{2^{p}}}\\&
\times\left[ \textbf{ber}(g^{2p}(|A^{*}|))-l\left(|[g^{2}(|A^{*}|)-\|g(|A^{*}|)\|^{2}]|^{p}\right)\right]^{\frac{1}{2^{p}}}
\end{align*}
for all $p\geq2$.
\end{theorem}
\begin{proof}
If $\hat{k}_{\lambda},\hat{k}_{\mu}\in {\mathcal H}$, then \eqref{al3} implies that
\begin{align*}
|\langle AB \hat{k}_{\lambda},\hat{k}_{\mu}\rangle|^{2p}&\leq r^{2p}(B)
\left[\langle f^{2p}(|A|)\hat{k}_{\lambda},\hat{k}_{\lambda}\rangle-\langle f^{2}(|A|)-\langle f^{2}(|A|)\hat{k}_{\lambda},\hat{k}_{\lambda}\rangle I|^{p}\hat{k}_{\lambda},\hat{k}_{\lambda}\rangle\right]\\&\times
\left[\langle g^{2p}(|A^{*}|)\hat{k}_{\mu},\hat{k}_{\mu}\rangle-\langle g^{2}(|A^{*}|)-\langle g^{2}(|A^{*}|)\hat{k}_{\mu},\hat{k}_{\mu}\rangle I|^{p}\hat{k}_{\mu},\hat{k}_{\mu}\rangle\right]\\&
\leq r^{2p}(B)
\left[\textbf{ber}(f^{2p}(|A|))-\langle f^{2}(|A|)-\langle f^{2}(|A|)\hat{k}_{\lambda},\hat{k}_{\lambda}\rangle I|^{p}\hat{k}_{\lambda},\hat{k}_{\lambda}\rangle\right]\\&\times
\left[\textbf{ber}(g^{2p}(|A^{*}|))-\langle g^{2}(|A^{*}|)-\langle g^{2}(|A^{*}|)\hat{k}_{\mu},\hat{k}_{\mu}\rangle I|^{p}\hat{k}_{\mu},\hat{k}_{\mu}\rangle\right].
\end{align*}
Now, let $\hat{k}_{\lambda}=\hat{k}_{\mu}$ and taking supremum over $\lambda\in \Omega$, we have
\begin{align*}
\textbf{ber}^{2p}(AB)&\leq r^{2p}(B)\left[ \textbf{ber}(f^{2p}(|A|))-l\left(|[f^{2}(|A|)-\|f(|A|)\|^{2}]|^{p}\right)\right]\\&
\times\left[ \textbf{ber}(g^{2p}(|A^{*}|))-l\left(|[g^{2}(|A^{*}|)-\|g(|A^{*}|)\|^{2}]|^{p}\right)\right].
\end{align*}
Now inequality \eqref{ki1} implies the statement.
\end{proof}
 Through following we state some refinements of Theorems \ref{th1} and \ref{th5}, which based on a refinement the Young inequality that is shown in \cite{kitta} by Kittaneh as follows:
\begin{align}\label{kitta}
a^{\alpha}b^{1-\alpha}\leq \alpha a+(1-\alpha)b-r_{0}(a^{1/2}-b^{1/2})^{2}
\end{align}
for any $a,b>0$, $0\leq \alpha\leq1$ and $r_{0}=\min\{\alpha, 1-\alpha\}$.
\begin{theorem}
Let $A,B\in {\mathcal B}(\mathcal H)$ such that $|A|B=B^{*}|A|$. If $f$ and $g$ are nonnegative continuous functions on $[0,\infty)$ satisfying $f(t)g(t)=t\,(t\geq 0)$, then
\begin{align*}
\textbf{ber}(AB)&\leq\frac{1}{2}r(B)\big(\textbf{ber}[f^{2}(|A|)+g^{2}(|A^{*}|)]-(\langle f^{2}(|A|)\hat{k}_{\lambda},\hat{k}_{\lambda}\rangle^{1/2}-\langle g^{2}(|A^{*}|)\hat{k}_{\lambda},\hat{k}_{\lambda}\rangle^{1/2})^{2}\big).
\end{align*}
In particular, for $f(t)=t^\alpha$ and $g(t)=t^{1-\alpha}$, which $0\leq\alpha\leq1$, we have
\begin{align*}
\textbf{ber}(AB)&\leq\frac{1}{2}r(B)\big(\textbf{ber}(|A|^{2\alpha}+|A^{*}|^{2(1-\alpha)})-(\langle |A|^{2\alpha}\hat{k}_{\lambda},\hat{k}_{\lambda}\rangle^{1/2}-\langle |A^{*}|^{2(1-\alpha)}\hat{k}_{\lambda},\hat{k}_{\lambda}\rangle^{1/2})^{2}\big).
\end{align*}
\end{theorem}
\begin{proof}
If we put $x=y=\hat{k}_{\lambda}$ in \eqref{al1} and applying \eqref{kitta}, we have
\begin{align}
&|\langle AB \hat{k}_{\lambda},\hat{k}_{\lambda}\rangle|\\&\leq r(B)\|f(|A|)\hat{k}_{\lambda}\|\|g(|A^{*}|)\hat{k}_{\lambda}\|\nonumber\\&
=r(B)\langle f^{2}(|A|)\hat{k}_{\lambda},\hat{k}_{\lambda}\rangle^{1/2}\langle g^{2}(|A^{*}|)\hat{k}_{\lambda},\hat{k}_{\lambda}\rangle^{1/2}\nonumber\\&
\leq\frac{1}{2}r(B)(\langle f^{2}(|A|)\hat{k}_{\lambda},\hat{k}_{\lambda}\rangle+\langle g^{2}(|A^{*}|)\hat{k}_{\lambda},\hat{k}_{\lambda}\rangle\nonumber
\\&\qquad-(\langle f^{2}(|A|)\hat{k}_{\lambda},\hat{k}_{\lambda}\rangle^{1/2}-\langle g^{2}(|A^{*}|)\hat{k}_{\lambda},\hat{k}_{\lambda}\rangle^{1/2})^{2})
\nonumber\\&
=\frac{1}{2}r(B)(\langle (f^{2}(|A|)+g^{2}(|A^{*}|))\hat{k}_{\lambda},\hat{k}_{\lambda}\rangle-(\langle f^{2}(|A|)\hat{k}_{\lambda},\hat{k}_{\lambda}\rangle^{1/2}-\langle g^{2}(|A^{*}|)\hat{k}_{\lambda},\hat{k}_{\lambda}\rangle^{1/2})^{2})\nonumber\\&
\leq\frac{1}{2}r(B)(\textbf{ber}[f^{2}(|A|)+g^{2}(|A^{*}|)]-(\langle f^{2}(|A|)\hat{k}_{\lambda},\hat{k}_{\lambda}\rangle^{1/2}-\langle g^{2}(|A^{*}|)\hat{k}_{\lambda},\hat{k}_{\lambda}\rangle^{1/2})^{2}).
\end{align}
By taking the supremum over $\lambda\in\Omega$ we get the desired inequality.
\end{proof}
\begin{theorem}
Let $A,B\in {\mathcal B}({\mathcal H})$ such that $|A|B=B^{*}|A|$. If $f$ and $g$ are nonnegative continuous functions on $[0,\infty)$ satisfying $f(t)g(t)=t\,(t\geq 0)$, then
{\begin{align*}
\textbf{ber}^{p}(AB)&\leq r^{p}(B)\Big[ \textbf{ber}(\frac{1}{\alpha}f^{\alpha p}(|A|)+\frac{1}{\beta}g^{\beta p}(|A^{*}|))\\&\qquad\qquad\qquad\qquad-r_{0}(\langle f^{2}(|A|)\hat{k}_{\lambda},\hat{k}_{\lambda}\rangle^{ \alpha p/4} -\langle g^{2}(|A^{*}|)\hat{k}_{\lambda},\hat{k}_{\lambda}\rangle^{\beta p/4})^{2}\Big].
\end{align*}}
for every $p\geq 1, \alpha\geq\beta>1$ with $\frac{1}{\alpha}+\frac{1}{\beta}=1$, $\beta p\geq2$ and $r_{0}=\min\{\frac{1}{\alpha},\frac{1}{\beta}\}$.
\end{theorem}
\begin{proof}
Let $\hat{k}_{\lambda}\in {\mathcal H}$. We have
{\begin{align*}
&|\langle AB\hat{k}_{\lambda},\hat{k}_{\lambda}\rangle|^{p}\\&\leq r^{p}(B)\|f(|A|)\hat{k}_{\lambda}\|^{p}\|g(|A^{*}|)\hat{k}_{\lambda}\|^{p}\\&
=r^{p}(B)\langle f^{2}(|A|)\hat{k}_{\lambda},\hat{k}_{\lambda}\rangle^{p/2}\langle g^{2}(|A^{*}|)\hat{k}_{\lambda},\hat{k}_{\lambda}\rangle^{p/2}\\&
\leq r^{p}(B)\left[\frac{1}{\alpha}\langle f^{2}(|A|)\hat{k}_{\lambda},\hat{k}_{\lambda}\rangle^{\alpha p/2}+\frac{1}{\beta}\langle g^{2}(|A^{*}|)\hat{k}_{\lambda},\hat{k}_{\lambda}\rangle^{\beta p/2}\right]
\end{align*}}
{\footnotesize\begin{align*}
&\leq r^{p}(B)\big[\frac{1}{\alpha}\langle f^{2}(|A|)\hat{k}_{\lambda},\hat{k}_{\lambda}\rangle^{\alpha p/2}+\frac{1}{\beta}\langle g^{2}(|A^{*}|)\hat{k}_{\lambda},\hat{k}_{\lambda}\rangle^{\beta p/2}\\&\qquad\qquad\qquad\qquad\qquad-r_{0}(\langle f^{2}(|A|)\hat{k}_{\lambda},\hat{k}_{\lambda}\rangle^{ \alpha p/4} -\langle g^{2}(|A^{*}|)\hat{k}_{\lambda},\hat{k}_{\lambda}\rangle^{\beta p/4})^{2} \big]\\&
\leq r^{p}(B)\Big[\frac{1}{\alpha}\langle f^{\alpha p}(|A|)\hat{k}_{\lambda},\hat{k}_{\lambda}\rangle+\frac{1}{\beta}\langle g^{\beta p}(|A^{*}|)\hat{k}_{\lambda},\hat{k}_{\lambda}\rangle\\&\qquad\qquad\qquad\qquad\qquad -r_{0}(\langle f^{2}(|A|)\hat{k}_{\lambda},\hat{k}_{\lambda}\rangle^{ \alpha p/4} -\langle g^{2}(|A^{*}|)\hat{k}_{\lambda},\hat{k}_{\lambda}\rangle^{\beta p/4})^{2}\Big]\\&
\leq r^{p}(B)\Big[\frac{1}{\alpha}\langle f^{\alpha p}(|A|)\hat{k}_{\lambda},\hat{k}_{\lambda}\rangle+\frac{1}{\beta}\langle g^{\beta p}(|A^{*}|)\hat{k}_{\lambda},\hat{k}_{\lambda}\rangle\\&\qquad\qquad\qquad\qquad\qquad -r_{0}(\langle f^{2}(|A|)\hat{k}_{\lambda},\hat{k}_{\lambda}\rangle^{ \alpha p/4} -\langle g^{2}(|A^{*}|)\hat{k}_{\lambda},\hat{k}_{\lambda}\rangle^{\beta p/4})^{2} \Big]\\&
=r^{p}(B)\langle [\frac{1}{\alpha}f^{\alpha p}(|A|)+\frac{1}{\beta}g^{\beta p}(|A^{*}|)]\hat{k}_{\lambda},\hat{k}_{\lambda}\rangle -r_{0}(\langle f^{2}(|A|)\hat{k}_{\lambda},\hat{k}_{\lambda}\rangle^{ \alpha p/4} -\langle g^{2}(|A^{*}|)\hat{k}_{\lambda},\hat{k}_{\lambda}\rangle^{\beta p/4})^{2}.\\&
\leq r^{p}(B)\left[\textbf{ber}(\frac{1}{\alpha}f^{\alpha p}(|A|)+\frac{1}{\beta}g^{\beta p}(|A^{*}|))-r_{0}(\langle f^{2}(|A|)\hat{k}_{\lambda},\hat{k}_{\lambda}\rangle^{ \alpha p/4} -\langle g^{2}(|A^{*}|)\hat{k}_{\lambda},\hat{k}_{\lambda}\rangle^{\beta p/4})^{2}\right].
\end{align*}}
By taking the supremum over $\lambda\in\Omega$ we get the desired result.
\end{proof}
\section{Number Berezin inequalities involving off diagonal matrices}
In this section, we improve and extend some Berezin number inequalities for $2\times2$ off diagonal matrices by nonnegative increasing convex functions. We  recall that the polarization identity says that,
\begin{align}\label{111}
\langle x,y\rangle=\frac{1}{4}\sum_{i=1}^{3}i^{k}\|x+i^{k}y\|^{2}\qquad(x,y\in \mathcal{H}).
\end{align}
For our goals, we need to the following lemmas.
\begin{lemma}\cite{Ba}\label{9}
Let  $A\in {\mathcal B}({\mathcal H_1})$, $B\in {\mathcal B}({\mathcal H_2}, {\mathcal H_1})$, $C\in {\mathcal B}({\mathcal H_1},{\mathcal H_2})$ and $D\in {\mathcal B}({\mathcal H_2})$. Then the following statements hold:\\
$(a)\,\,\textbf{ber}\left(\left[\begin{array}{cc}
 A&0\\
 0&D
\end{array}\right]\right)$
$\leq \max \{\textbf{ber}(A), \textbf{ber}(D)\};$
\\
\\
$(b)\,\,\textbf{ber}\left(\left[\begin{array}{cc}
0&B\\
C&0
\end{array}\right]\right)$
$\leq$ $ \frac{1}{2}(\|B\|+\|C\|);$\\
$(c)\,\,\textbf{ber}(A)=\sup_{\theta\in \mathbb{R}}\textbf{ber}(Re(e^{i\theta}A))$.
\end{lemma}
\begin{lemma}\cite{sheb}\label{convex}
Let h be a nonnegative nondecreasing convex function on $[0,\infty)$ and let $A,B\in \mathcal{B}(\mathcal{H})$ be positive operators. Then
\begin{align*}
h\left( \left\|\frac{A+B}{2}\right\|\right)\leq\left\|\frac{h(A)+h(B)}{2}\right\|.
\end{align*}
\end{lemma}
\begin{theorem}
Let
$T=\left[\begin{array}{cc}
 0&B\\
 C&0
 \end{array}\right]\in {\mathcal B}(\mathcal H(\Omega_{1})\oplus\mathcal H(\Omega_{2}))$ and  $f$, $g$ be nonnegative  continuous  functions on $[0, \infty)$ satisfying the relation $f(t)g(t)=t\,\,(t\in [0, \infty))$. Then
 \begin{align}\label{7}
h(\textbf{ber}(T))\leq \frac{1}{4}\|h(f^{2}(|C|))+h(g^{2}(|C|))\|+\frac{1}{4}\|h(f^{2}(|B|))+h(g^{2}(|B|))\|.
\end{align}
\end{theorem}
\begin{proof}
Let $B=U|B|$ and $C=V|C|$ be the polar decomposition of operators $B$ and $C$. Then $T=W|T|=\left[\begin{array}{cc}
 0&U\\
 V&0
 \end{array}\right]\left[\begin{array}{cc}
 |C|&0\\
 0&|B|
 \end{array}\right]$ is the polar decomposition of $T$.\\
For any $(\lambda_{1},\lambda_{2})\in \Omega_{1}\times\Omega_{2}$, let $\hat{k}_{(\lambda_{1},\lambda_{2})}=\left[\begin{array}{cc}
 k_{\lambda_{1}}\\
 k_{\lambda_{2}}
 \end{array}\right]$ be the normalized reproducing kernel in $\mathcal H(\Omega_{1})\oplus\mathcal H(\Omega_{2})$. Then
 \begin{align*}
 \langle &\texttt{Re}\,e^{i\theta}T\hat{k}_{(\lambda_{1},\lambda_{2})},\hat{k}_{(\lambda_{1},\lambda_{2})}\rangle\\&=\texttt{Re}\langle e^{i\theta}W|T|\hat{k}_{(\lambda_{1},\lambda_{2})},\hat{k}_{(\lambda_{1},\lambda_{2})}\rangle\\&
 =\texttt{Re}\langle e^{i\theta}W f(|T|)g(|T|)\hat{k}_{(\lambda_{1},\lambda_{2})},\hat{k}_{(\lambda_{1},\lambda_{2})}\rangle\\&
 =\texttt{Re}\langle e^{i\theta}g(|T|)\hat{k}_{(\lambda_{1},\lambda_{2})},f(|T|)W^{*}\hat{k}_{(\lambda_{1},\lambda_{2})}\rangle\\&
 =\texttt{Re}\left\langle e^{i\theta}\left[\begin{array}{cc}
 g(|C|)&0\\
 0&g(|B|)
 \end{array}\right]\left[\begin{array}{cc}
 k_{\lambda_{1}}\\
 k_{\lambda_{2}}
 \end{array}\right],\left[\begin{array}{cc}
 f(|C|)&0\\
 0&f(|B|)
 \end{array}\right]\left[\begin{array}{cc}
 0&V^{*}\\
 U^{*}&0
 \end{array}\right]\left[\begin{array}{cc}
 k_{\lambda_{1}}\\
 k_{\lambda_{2}}
 \end{array}\right]\right\rangle\\&
 =\texttt{Re} \langle e^{i\theta}(g(|C|)k_{\lambda_{1}},g(|B|)k_{\lambda_{2}}),(f(|C|)V^{*}k_{\lambda_{2}},f(|B|)U^{*}k_{\lambda_{1}})\rangle\\&
 =\texttt{Re} (\langle e^{i\theta}g(|C|)k_{\lambda_{1}},f(|C|)V^{*}k_{\lambda_{2}}\rangle+\langle e^{i\theta}g(|B|)k_{\lambda_{2}},f(|B|)U^{*}k_{\lambda_{1}}\rangle)\\&
 =\frac{1}{4}\left( \|e^{i\theta}g(|C|)k_{\lambda_{1}}+f(|C|)V^{*}k_{\lambda_{2}}\|^{2}-\|e^{i\theta}g(|C|)k_{\lambda_{1}}-f(|C|)V^{*}k_{\lambda_{2}}\|^{2}\right)\\&
 +\frac{1}{4}\left( \|e^{i\theta}g(|B|)k_{\lambda_{2}}+f(|B|)U^{*}k_{\lambda_{1}}\|^{2}-\|e^{i\theta}g(|B|)k_{\lambda_{2}}-f(|B|)U^{*}k_{\lambda_{1}}\|^{2}\right)\\&
 \qquad\qquad\qquad (\textrm {by \eqref{111}})\\&
 \leq \frac{1}{4} \|e^{i\theta}g(|C|)k_{\lambda_{1}}+f(|C|)V^{*}k_{\lambda_{2}}\|^{2}+\frac{1}{4} \|e^{i\theta}g(|B|)k_{\lambda_{2}}+f(|B|)U^{*}k_{\lambda_{1}}\|^{2}\\&
 =\frac{1}{4} \Big\|[e^{i\theta}g(|C|)  f(|C|)V^{*}]\left[\begin{array}{cc}
 k_{\lambda_{1}}\\
 k_{\lambda_{2}}
 \end{array}\right]\Big\|^{2}+\frac{1}{4} \Big\|[e^{i\theta}g(|B|)  f(|B|)U^{*}\left[\begin{array}{cc}
 k_{\lambda_{1}}\\
 k_{\lambda_{2}}
 \end{array}\right]\Big\|^{2}\\&
 \leq\frac{1}{4}\|[e^{i\theta}g(|C|)  f(|C|)V^{*}]\|^{2}+\frac{1}{4} \|[e^{i\theta}g(|B|)  f(|B|)U^{*}\|^{2}\\&
 =\frac{1}{4} \Big\|[e^{i\theta}g(|C|)  f(|C|)V^{*}]\left[\begin{array}{cc}
 e^{i\theta}g(|C|)\\
 Vf(|C|)
 \end{array}\right]\Big\|+\frac{1}{4} \Big\|[e^{i\theta}g(|B|)  f(|B|)U^{*}\left[\begin{array}{cc}
 Uf(|B|)\\
 e^{-i\theta}g(|B|)
 \end{array}\right]\Big\|
 \end{align*}
 \begin{align*}
&=\frac{1}{4}\|g^{2}(|C|)+f(|C|)V^{*}Vf(|C|)\|+\frac{1}{4}\|f(|B|)U^{*}Uf(|B|)+g^{2}(|B|)\|\\&
 =\frac{1}{4}\|f^{2}(|C|)+g^{2}(|C|)\|+\frac{1}{4}\|f^{2}(|B|)+g^{2}(|B|)\|.
 \end{align*}
 By taking the supremun over all $\lambda \in \Omega$, Lemma \ref{9}(c) and applying Lemma \ref{convex} for any nondecreasing convex function $h$, we get the desired result.
\end{proof}
\begin{corollary}
Let $T=\left[\begin{array}{cc}
 0&B\\
 C&0
 \end{array}\right]\in {\mathcal B}(\mathcal H(\Omega_{1})\oplus\mathcal H(\Omega_{2}))$. Then for any $\alpha\in [0,1]$ and $p\geq1$,
 \begin{align*}
 \textbf{ber}^{p}\left(\left[\begin{array}{cc}
 0&B\\
 C&0
 \end{array}\right]\right)\leq\frac{1}{4}\||B|^{2p\alpha}+|B|^{2p(1-\alpha)}\|+\frac{1}{4}\||C|^{2p\alpha}+|C|^{2p(1-\alpha)}\|.
 \end{align*}
\end{corollary}
\begin{proof}
By putting $h(t)=t^{p}$, $f(t)=t^{\alpha}$ and $g(t)=t^{1-\alpha}$ in inequality \eqref{7}, we get the desired inequality.
\end{proof}

%
\section{Berezin number and Cartesian decomposition}
In this section, our purpose is to give an upper bound for Berezin number in
terms of the Cartesian decomposition of operators on a RKHS $\mathcal{H=H(}%
\Omega\mathcal{)}$. Before giving the results, we need several well known lemmas.

\begin{lemma}\label{p1}\cite{mc}Let $A\in\mathcal{B}\left(  \mathcal{H}\right)  $ be a positive
operator. Then for $x\in\mathcal{H}$\newline$\left(
i\right)  $ $\left\langle A^{p}x,x\right\rangle \geq\left\Vert x\right\Vert
^{2\left(  1-p\right)  }\left\langle Ax,x\right\rangle ^{p}$, if $p\geq
1;$\newline$\left(  ii\right)  $ $\left\langle A^{p}x,x\right\rangle
\leq\left\Vert x\right\Vert ^{2\left(  1-p\right)  }\left\langle
Ax,x\right\rangle ^{p}$, if $0<p<1$.
\end{lemma}

\begin{lemma}\label{p2}\cite{kit}Let $A\in\mathcal{B}\left(  \mathcal{H}\right)  $ and
$0\leq p\leq1$. Then for $x,y\in\mathcal{H}$%
\[
\left\vert \left\langle Ax,y\right\rangle \right\vert ^{2}\leq\left\langle
\left\vert A\right\vert ^{2p}x,x\right\rangle \left\langle \left\vert A^{\ast
}\right\vert ^{2\left(  1-p\right)  }y,y\right\rangle \text{.}%
\]
\end{lemma}

\begin{lemma}\label{p3}\cite{bohr}Let $x_{n}$ be a positive real number, $1\leq n\leq k$. Then for
each $p\geq1$%
\[
\left(
{\displaystyle\sum\limits_{n=1}^{k}}
x_{n}\right)  ^{p}\leq k^{p-1}%
{\displaystyle\sum\limits_{n=1}^{k}}
x_{n}^{p}.
\]
\end{lemma}
Now, we are ready to give our results.
\begin{theorem}
Let $A_{n}\in\mathcal{B}\left(  \mathcal{H}\right)  $ have the Cartesian
decomposition $A_{n}=B_{n}+iC_{n}$ for $n=1,...,k$ and $p\geq1$. Then%
\[
\mathbf{ber}^{p}\left(
{\displaystyle\sum\limits_{n=1}^{k}}
A_{n}\right)  \leq\left(  \sqrt{2}k\right)  ^{p-1}\sup_{\lambda\in\Omega
}\left[
{\displaystyle\sum\limits_{n=1}^{k}}
\left(  \widetilde{\left\vert B_{n}\right\vert ^{2p}}\left(  \lambda\right)
+\widetilde{\left\vert C_{n}\right\vert ^{2p}}\left(  \lambda\right)  \right)
^{\frac{1}{2}}\right]
\]
for $\lambda\in\Omega$.
\end{theorem}

\begin{proof}
Let $\widehat{k}_{\lambda}\in\mathcal{H}\left(  \Omega\right)  $. Then%
\[
\left\vert \left\langle
{\displaystyle\sum\limits_{n=1}^{k}}
A_{n}\widehat{k}_{\lambda},\widehat{k}_{\lambda}\right\rangle \right\vert
^{p}\leq\left(
{\displaystyle\sum\limits_{n=1}^{k}}
\left(  \left\langle B_{n}\widehat{k}_{\lambda},\widehat{k}_{\lambda
}\right\rangle ^{2}+\left\langle C_{n}\widehat{k}_{\lambda},\widehat
{k}_{\lambda}\right\rangle ^{2}\right)  ^{\frac{1}{2}}\right)  ^{p}%
\]
for $\lambda\in\Omega$. Applying Lemma \ref{p2} for $\alpha=1$, we get%
\[
\left\vert \left\langle
{\displaystyle\sum\limits_{n=1}^{k}}
A_{n}\widehat{k}_{\lambda},\widehat{k}_{\lambda}\right\rangle \right\vert
^{p}\leq\left(
{\displaystyle\sum\limits_{n=1}^{k}}
\left(  \left\langle \left\vert B_{n}\right\vert ^{2}\widehat{k}_{\lambda
},\widehat{k}_{\lambda}\right\rangle +\left\langle \left\vert C_{n}\right\vert
^{2}\widehat{k}_{\lambda},\widehat{k}_{\lambda}\right\rangle \right)
^{\frac{1}{2}}\right)  ^{p}%
\]
for $\lambda\in\Omega$. Using  Lemma \ref{p3} and Lemma \ref{p1}, \ we obtain%
\begin{align*}
\left\vert \left\langle
{\displaystyle\sum\limits_{n=1}^{k}}
A_{n}\widehat{k}_{\lambda},\widehat{k}_{\lambda}\right\rangle \right\vert
^{p}  & \leq k^{p-1}%
{\displaystyle\sum\limits_{n=1}^{k}}
\left(  \left\langle \left\vert B_{n}\right\vert ^{2}\widehat{k}_{\lambda
},\widehat{k}_{\lambda}\right\rangle +\left\langle \left\vert C_{n}\right\vert
^{2}\widehat{k}_{\lambda},\widehat{k}_{\lambda}\right\rangle \right)
^{\frac{p}{2}}\\
& \leq\left(  \sqrt{2}k\right)  ^{p-1}%
{\displaystyle\sum\limits_{n=1}^{k}}
\left(  \left\langle \left\vert B_{n}\right\vert ^{2}\widehat{k}_{\lambda
},\widehat{k}_{\lambda}\right\rangle ^{p}+\left\langle \left\vert
C_{n}\right\vert ^{2}\widehat{k}_{\lambda},\widehat{k}_{\lambda}\right\rangle
^{p}\right)  ^{\frac{1}{2}}\\
& \leq\left(  \sqrt{2}k\right)  ^{p-1}%
{\displaystyle\sum\limits_{n=1}^{k}}
\left(  \left\langle \left\vert B_{n}\right\vert ^{2p}\widehat{k}_{\lambda
},\widehat{k}_{\lambda}\right\rangle +\left\langle \left\vert C_{n}\right\vert
^{2p}\widehat{k}_{\lambda},\widehat{k}_{\lambda}\right\rangle \right)
^{\frac{1}{2}}%
\end{align*}
for $\lambda\in\Omega$. Taking supremum over $\lambda\in\Omega$, we have%
\[
\sup_{\lambda\in\Omega}\left\vert \left\langle
{\displaystyle\sum\limits_{n=1}^{k}}
A_{n}\widehat{k}_{\lambda},\widehat{k}_{\lambda}\right\rangle \right\vert
^{p}\leq\left(  \sqrt{2}k\right)  ^{p-1}\sup_{\lambda\in\Omega}\left[
{\displaystyle\sum\limits_{n=1}^{k}}
\left(  \left\langle \left\vert B_{n}\right\vert ^{2p}\widehat{k}_{\lambda
},\widehat{k}_{\lambda}\right\rangle +\left\langle \left\vert C_{n}\right\vert
^{2r}\widehat{k}_{\lambda},\widehat{k}_{\lambda}\right\rangle \right)
^{\frac{1}{2}}\right]
\]
and so%
\[
\mathbf{ber}^{p}\left(
{\displaystyle\sum\limits_{n=1}^{k}}
A_{n}\right)  \leq\left(  \sqrt{2}k\right)  ^{p-1}\sup_{\lambda\in\Omega
}\left[
{\displaystyle\sum\limits_{n=1}^{k}}
\left(  \widetilde{\left\vert B_{n}\right\vert ^{2p}}\left(  \lambda\right)
+\widetilde{\left\vert C_{n}\right\vert ^{2p}}\left(  \lambda\right)  \right)
^{\frac{1}{2}}\right]  .
\]

\end{proof}

\begin{theorem}
Let $A_{n}\in\mathcal{B}\left(  \mathcal{H}\right)  $ have the Cartesian
decomposition $A_{n}=B_{n}+iC_{n}$ for $n=1,...,k$ and $p\geq1$. Then%
\[
\mathbf{ber}^{p}\left(
{\displaystyle\sum\limits_{n=1}^{k}}
A_{n}\right)  \leq k^{p-1}2^{\frac{p}{2}-1}\sup_{\lambda\in\Omega}\left[
{\displaystyle\sum\limits_{n=1}^{k}}
\left(  \widetilde{\left\vert B_{n}+C_{n}\right\vert ^{2p}}\left(
\lambda\right)  +\widetilde{\left\vert B_{n}-C_{n}\right\vert ^{2p}}\left(
\lambda\right)  \right)  ^{\frac{1}{2}}\right]
\]
for $\lambda\in\Omega$.
\end{theorem}

\begin{proof}
Let $\widehat{k}_{\lambda}\in\mathcal{H}\left(  \Omega\right)  $. We have%
\begin{align*}
\left\vert \left\langle
{\displaystyle\sum\limits_{n=1}^{k}}
A_{n}\widehat{k}_{\lambda},\widehat{k}_{\lambda}\right\rangle \right\vert
^{p}  & \leq\left(
{\displaystyle\sum\limits_{n=1}^{k}}
\left(  \left\langle B_{n}\widehat{k}_{\lambda},\widehat{k}_{\lambda
}\right\rangle ^{2}+\left\langle C_{n}\widehat{k}_{\lambda},\widehat
{k}_{\lambda}\right\rangle ^{2}\right)  ^{\frac{1}{2}}\right)  ^{p}\\
& \leq\left(
{\displaystyle\sum\limits_{n=1}^{k}}
\left(  \frac{1}{2}\left(  \left\langle \left(  B_{n}+C_{n}\right)
\widehat{k}_{\lambda},\widehat{k}_{\lambda}\right\rangle ^{2}+\left\langle
\left(  B_{n}-C_{n}\right)  \widehat{k}_{\lambda},\widehat{k}_{\lambda
}\right\rangle ^{2}\right)  \right)  ^{\frac{1}{2}}\right)  ^{p}%
\end{align*}
for $\lambda\in\Omega$. Using Lemma \ref{p3} and Lemma \ref{p2}, respectively, we
have%
\begin{align*}
\left\vert \left\langle
{\displaystyle\sum\limits_{n=1}^{k}}
A_{n}\widehat{k}_{\lambda},\widehat{k}_{\lambda}\right\rangle \right\vert
^{p}  & \leq k^{p-1}2^{-\frac{p}{2}}%
{\displaystyle\sum\limits_{n=1}^{k}}
\left(  \left\langle \left(  B_{n}+C_{n}\right)  \widehat{k}_{\lambda
},\widehat{k}_{\lambda}\right\rangle ^{2}+\left\langle \left(  B_{n}%
-C_{n}\right)  \widehat{k}_{\lambda},\widehat{k}_{\lambda}\right\rangle
^{2}\right)  ^{\frac{p}{2}}\\
& \leq k^{p-1}2^{-\frac{p}{2}}%
{\displaystyle\sum\limits_{n=1}^{k}}
\left(  \left\langle \left\vert B_{n}+C_{n}\right\vert ^{2}\widehat
{k}_{\lambda},\widehat{k}_{\lambda}\right\rangle +\left\langle \left\vert
B_{n}-C_{n}\right\vert ^{2}\widehat{k}_{\lambda},\widehat{k}_{\lambda
}\right\rangle \right)  ^{\frac{p}{2}}%
\end{align*}
for $\lambda\in\Omega$. Then, applying Lemma \ref{p3} and Lemma \ref{p1}, we obtain%
\begin{align*}
\left\vert \left\langle
{\displaystyle\sum\limits_{n=1}^{k}}
A_{n}\widehat{k}_{\lambda},\widehat{k}_{\lambda}\right\rangle \right\vert
^{p}  & \leq k^{p-1}2^{\frac{p}{2}-1}%
{\displaystyle\sum\limits_{n=1}^{k}}
\left(  \left\langle \left\vert B_{n}+C_{n}\right\vert ^{2}\widehat
{k}_{\lambda},\widehat{k}_{\lambda}\right\rangle ^{p}+\left\langle \left\vert
B_{n}-C_{n}\right\vert ^{2}\widehat{k}_{\lambda},\widehat{k}_{\lambda
}\right\rangle ^{p}\right)  ^{\frac{1}{2}}\\
& \leq k^{p-1}2^{\frac{p}{2}-1}%
{\displaystyle\sum\limits_{n=1}^{k}}
\left(  \left\langle \left\vert B_{n}+C_{n}\right\vert ^{2p}\widehat
{k}_{\lambda},\widehat{k}_{\lambda}\right\rangle +\left\langle \left\vert
B_{n}-C_{n}\right\vert ^{2p}\widehat{k}_{\lambda},\widehat{k}_{\lambda
}\right\rangle \right)  ^{\frac{1}{2}}%
\end{align*}
for $\lambda\in\Omega$. Taking supremum over $\lambda\in\Omega$, we reach that%
\[
\mathbf{ber}^{p}\left(
{\displaystyle\sum\limits_{n=1}^{k}}
A_{n}\right)  \leq k^{p-1}2^{\frac{p}{2}-1}\sup_{\lambda\in\Omega}\left[
{\displaystyle\sum\limits_{n=1}^{k}}
\left(  \widetilde{\left\vert B_{n}+C_{n}\right\vert ^{2p}}\left(
\lambda\right)  +\widetilde{\left\vert B_{n}-C_{n}\right\vert ^{2p}}\left(
\lambda\right)  \right)  ^{\frac{1}{2}}\right]  .
\]
\end{proof}
\bigskip
\bibliographystyle{amsplain}

\end{document}